\documentclass[12pt]{amsart}

\usepackage{amsmath,amssymb,amsfonts,amsthm}
\usepackage{mathtools}
\usepackage[margin=1in]{geometry}

\usepackage{graphicx}
\usepackage{tikz}
\usetikzlibrary{arrows.meta,graphs,shapes.geometric,decorations.pathreplacing}
\usepackage{pgfplots}
\pgfplotsset{compat=1.18}

\usepackage{booktabs}
\usepackage{multicol}
\usepackage{enumitem}

\usepackage[numbers,sort&compress]{natbib}

\usepackage{xcolor}
\definecolor{myteal}{RGB}{0,128,128}
\usepackage[pdfborder={0 0 0}]{hyperref}
\hypersetup{
  colorlinks=true,
  citecolor=myteal,
  linkcolor=blue,
  urlcolor=myteal
}

\usepackage{aliascnt}
\usepackage[capitalise, noabbrev]{cleveref}

\usepackage{parskip}

\usepackage{listings}
\newcommand{\newaliastheorem}[3]{%
  \newaliascnt{#1}{thm}%
  \newtheorem{#1}[#1]{#2}%
  \aliascntresetthe{#1}%
  \crefname{#1}{#2}{#3}%
  \Crefname{#1}{#2}{#3}%
}

\theoremstyle{plain}
\newtheorem{thm}{Theorem}[section]
\crefname{thm}{Theorem}{Theorems}
\Crefname{thm}{Theorem}{Theorems}

\newaliastheorem{lem}{Lemma}{Lemmas}
\newaliastheorem{prop}{Proposition}{Propositions}
\newaliastheorem{cor}{Corollary}{Corollaries}
\newaliastheorem{clm}{Claim}{Claims}
\newaliastheorem{conj}{Conjecture}{Conjectures}

\theoremstyle{definition}
\newaliastheorem{defn}{Definition}{Definitions}
\newaliastheorem{ex}{Example}{Examples}
\newaliastheorem{notation}{Notation}{Notations}
\newaliastheorem{setup}{Set-up}{Set-ups}
\newaliastheorem{remark}{Remark}{Remarks}
\newaliastheorem{problem}{Problem}{Problems}
\newaliastheorem{question}{Question}{Questions}
\newaliastheorem{observation}{Observation}{Observation}

\setlist[enumerate]{leftmargin=*,itemsep=2pt,topsep=4pt}
\SetEnumerateShortLabel{a}{\textup{(\alph*)}}
\SetEnumerateShortLabel{A}{\textup{(\Alph*)}}
\SetEnumerateShortLabel{1}{\textup{(\arabic*)}}
\SetEnumerateShortLabel{i}{\textup{(\roman*)}}
\SetEnumerateShortLabel{I}{\textup{(\Roman*)}}

\makeatletter
\newtheorem*{rep@theorem}{\rep@title}
\newcommand{\newreptheorem}[2]{%
\newenvironment{rep#1}[1]{%
 \def\rep@title{#2 \ref{##1}}%
 \begin{rep@theorem}}%
 {\end{rep@theorem}}}
\makeatother

\newreptheorem{theorem}{Theorem}

\DeclareMathOperator{\reg}{reg}
\DeclareMathOperator{\pdim}{pdim}
\DeclareMathOperator{\Hilb}{Hilb}
\DeclareMathOperator{\ord}{ord}
\DeclareMathOperator{\bight}{bight}
\DeclareMathOperator{\height}{ht}

\begin{document}

\title{Algebraic Invariants of Edge Ideals Under Suspension}

\author[Kara]{Selvi Kara}
\address[S.~Kara]{Department of Mathematics, Bryn Mawr College, Bryn Mawr, PA 19010}
\email{\href{mailto:skara@brynmawr.edu}{skara@brynmawr.edu}}

\author[Vien]{Dalena Vien}
\address[D.~Vien]{Department of Mathematics, Bryn Mawr College, Bryn Mawr, PA 19010}
\email{\href{mailto:dvien@brynmawr.edu}{dvien@brynmawr.edu}}

\begin{abstract}
The central question of this paper is: how do algebraic invariants of edge ideals change under natural graph operations?
We study this question through the lens of \emph{suspensions}.
The (full) suspension of a graph is obtained by adjoining a new vertex adjacent to every vertex of the original graph; this construction is well-understood in the literature. Motivated by the fact that regularity is preserved under full suspension while projective dimension becomes maximal, we refine the construction to \emph{selective suspensions}, where the new vertex is joined only to a prescribed subset of vertices. We focus on two extremal choices: minimal vertex covers and maximal independent sets.
For suspensions over minimal vertex covers of an arbitrary graph, regularity is preserved and projective dimension increases by one.
Moreover, the independence polynomial changes in a controlled way, allowing us to track $\mathfrak a$-invariants under cover suspension.
In contrast, the analogous uniform behavior fails in general for suspensions over maximal independent sets.
We therefore analyze paths and cycles and give a complete description: projective dimension always increases by one, and regularity and the $\mathfrak a$-invariant are preserved except for a unique extremal family of paths, where both invariants increase by one.
\end{abstract}

\maketitle

\section{Introduction}

Edge ideals provide a fertile playground for observing the following broad theme in combinatorial commutative algebra:
graphs encode algebraic information and structured changes to a graph can induce controlled changes in the homological invariants of the corresponding ideal.  Given a
finite simple graph $G$ on vertex set $V(G)=\{x_1,\dots,x_n\}$, its edge ideal
\[
I(G)=(x_ix_j:\ \{x_i,x_j\}\in E(G))\subseteq R=\Bbbk[x_1,\dots,x_n]
\]
is a squarefree monomial ideal. The minimal free resolution of $R/I(G)$ captures both the combinatorics of $G$
and the topology of its independence complex $\Delta(G)$; see, for instance,
\cite{MoreyVillarreal,herzog2011monomial,Hochster77,EagonReiner98}.  A guiding problem in the area is
to understand how graph operations affect graded Betti numbers and related algebraic invariants such as the
Castelnuovo--Mumford regularity $\reg(R/I(G))$ and projective dimension $\pdim(R/I(G))$.  This
perspective complements an extensive literature relating these invariants to graph parameters and
induced substructures; see, e.g., \cite{ProjDim,ha2008monomial,IndMatch,RegMatch,CWRegDeg}.

In this paper we study this problem through \emph{suspensions of graphs}.
The classical (full) suspension $\widehat{G}$ of a graph $G$ is obtained by adjoining a new vertex adjacent to every vertex of $G$.  Suspension provides a way to connect the toric edge ring of $\widehat{G}$ with the Rees algebra of the edge ideal of $G$ (see  \cite{herzog2011monomial, Hibi_suspension_toric}).
More generally,  the concept of \emph{$S$-suspension} $G^{S}$ for an independent set $S$, obtained by adjoining a new vertex adjacent to all vertices in $V(G)\setminus S$, was introduced in \cite{IndMatch}.
Equivalently, this is the construction obtained by adjoining a new vertex adjacent to a (not necessarily minimal) vertex cover.

Motivated by these constructions, we work in a common general framework.
Given a graph $G=(V,E)$ and a subset $\mathcal C\subseteq V$, we form a new graph $G(\mathcal C)$ by adjoining a vertex $z$ and connecting $z$ to every vertex of $\mathcal C$ (and to no other vertices).
We call $G(\mathcal C)$ the \emph{$\mathcal C$-suspension} of $G$ and refer to $z$ as the \emph{suspension vertex}.
If $\mathcal C$ is a vertex cover then $G(\mathcal C)$ agrees with the $S$-suspension of \cite{IndMatch} (take $S=V\setminus \mathcal C$).
On the ideal side, this corresponds to
\[
I\bigl(G(\mathcal C)\bigr) \;=\; I(G) + (z x_i : x_i \in \mathcal C) \subseteq S:=R[z].
\]
Our goal is to understand how this controlled addition of a single vertex with prescribed neighborhood
changes three graded invariants of the quotient: regularity, projective dimension, and the $\mathfrak a$-invariant.
The case $\mathcal{C}=V$ recovers the suspension $\widehat{G}$, which serves as a baseline for the behavior of these invariants.
We then turn to two canonical choices for $\mathcal{C}$ that sit at opposite extremes and are complementary:
minimal vertex covers and maximal independent sets.  Suspension over a minimal vertex cover produces a uniform
effect for all graphs, while suspension over a maximal independent set can be sensitive to the underlying graph.
Our results show that this sensitivity already appears in paths, where the
behavior is rigid except for a unique extremal configuration of paths.

In our study of suspension over minimal vertex covers, we prove that
regularity is preserved and projective dimension increases by one (\Cref{thm:cover-cone-reg-pdim}).  Moreover, the independence
polynomial changes by an explicit additive term, which lets us track the multiplicity $M(G)$ of $-1$ as a
root of the independence polynomial (concept introduced in \cite{biermann2026realizableregdeghpairs}) and hence compare $\mathfrak a$--invariants; see \Cref{prop:ainv-min-vertex-cover}. This type of suspension over vertex covers was studied in \cite{IndMatch} and it was shown that regularity is preserved under suspension over any vertex cover (not necessarily minimal). In our context, it is crucial to work with minimal vertex covers to prove our projective dimension result.

Outside the minimal vertex cover regime, suspension can depend strongly on the choice of $\mathcal{C}$.
We show this dependence already for maximal independent sets, and we give a complete analysis for
cycles and paths.  For cycles, suspending over any maximal independent set preserves regularity and the
$\mathfrak a$-invariant, and increases projective dimension by one (\Cref{thm:cycle_reg_pdim,{thm:cycle_a-inv}}).
For paths, projective dimension again always increases by one, while regularity and the
$\mathfrak a$-invariant are preserved except for a unique extremal configuration  (when $n\equiv 1 \pmod 3$ and $\mathcal C=\{x_1,x_4,\dots,x_{3k+1}\}$), where both
invariants increase by one (\Cref{thm:path_reg,thm:path_pdim,thm:path_a_inv}).

From the technical viewpoint, the proofs combine Hochster's formula and induced subcomplexes with two
complementary tools: Mayer--Vietoris for unions involving cones, and discrete Morse theory for
precise control of top homology in the cycle and path cases \cite{Forman98,KozlovBook,JonssonGraphs}.
From the conceptual viewpoint, selective suspension provides a tractable graph operation rich enough to
exhibit both rigidity and sharp thresholds, making it a useful approach for understanding how
local graph structure propagates to homological data.

\textbf{Organization.}
\Cref{sec:preliminaries} gathers the background material and relevant results from the literature needed in the sequel.
\Cref{sec:coning-operations} develops the full suspension and suspensions over vertex covers; in particular, it proves \Cref{thm:cover-cone-reg-pdim}, showing that suspending over a minimal vertex cover increases projective dimension by one.
\Cref{sec:cycles} is devoted to cycles and proves \Cref{thm:cycle_reg_pdim,thm:cycle_a-inv}; the \emph{wide-spoke} case (when $n\equiv 0 \pmod 3$ and the maximal independent set has size $\lceil \frac{n}{3} \rceil$) is handled via discrete Morse theory in \Cref{subsec:wide_spokes}.
\Cref{sec:paths} treats paths and proves \Cref{thm:path_reg,thm:path_pdim,thm:path_a_inv}, including the unique extremal configuration.
Finally, \Cref{sec:final} concludes with further directions and questions suggested by these results.

\section{Preliminaries}\label{sec:preliminaries}

Throughout, $G=(V,E)$ denotes a finite simple graph on the vertex set
$V=\{x_1,\dots,x_n\}$, and $R=\Bbbk[x_1,\dots,x_n]$ is the standard graded polynomial ring.
For $W\subseteq V$ we write $G|_W$ for the induced subgraph on $W$, and for $v\in V$ we write
$G-v:=G|_{V\setminus\{v\}}$. We use reduced homology $\widetilde H_\ast(-;\Bbbk)$ over a fixed
field~$\Bbbk$.

\subsection{Graphs, edge ideals, and independence complexes}

\begin{defn}\label{def:independent-sets}
A subset $X\subseteq V$ is \emph{independent} if its vertices are pairwise non-adjacent in $G$.
The \emph{independence number} of $G$ is
\[
\alpha(G):=\max\{|X| : X\subseteq V \text{ is independent}\}.
\]
\end{defn}

\begin{defn}\label{def:vertex-covers}
A subset $\mathcal{M} \subseteq V$ is a \emph{vertex cover} if every edge of $G$ is incident to at least one
vertex of $\mathcal{M} $. A vertex cover is \emph{minimal} if it contains no proper subset that is also a
vertex cover.
\end{defn}

The complement operation relates independent sets and vertex covers: $X\subseteq V$ is independent
if and only if $V\setminus X$ is a vertex cover. In particular, if $X$ is a \emph{maximal}
independent set, then $V\setminus X$ is a \emph{minimal} vertex cover.

\begin{defn}\label{def:edge-ideal}
The \emph{edge ideal} of $G$ is the squarefree monomial ideal
\[
I(G)=(x_ix_j : \{x_i,x_j\}\in E)\ \subseteq\ R.
\]
\end{defn}

\begin{defn}\label{def:bight}
The \emph{big height} of $I(G)$ is
\[
\bight I(G) := \max\{\height \mathfrak p : \mathfrak p \text{ is a minimal prime of } I(G)\}.
\]
Equivalently, $\bight I(G)$ is the maximum cardinality of a minimal vertex cover of~$G$.
\end{defn}

\begin{lem}\label{lem:dim-alpha}
One has $\dim(R/I(G))=\alpha(G)$.
\end{lem}

\begin{proof}
Minimal primes of $I(G)$ correspond to minimal vertex covers: if $C$ is a minimal vertex cover,
then $\mathfrak p_C:=(x_i : x_i\in C)$ is a minimal prime of $I(G)$, and every minimal prime arises
in this way. Hence
\[
\dim(R/I(G)) = n - \height(I(G)) = n - \min\{|C| : C \text{ is a vertex cover}\}.
\]
As noted above, the minimum vertex cover size equals $n-\alpha(G)$, so $\dim(R/I(G))=\alpha(G)$.
\end{proof}

\begin{defn}\label{def:independence-complex}
The \emph{independence complex} of $G$ is the simplicial complex
\[
\Delta(G) = \{X\subseteq V : X \text{ is independent in } G\}.
\]
\end{defn}

The minimal nonfaces of $\Delta(G)$ are precisely the edges of $G$. Hence,
the Stanley--Reisner ideal
of $\Delta(G)$ is $I_{\Delta(G)}=I(G)$. In particular,
\[
R/I(G)\ \cong\ \Bbbk[\Delta(G)].
\]
Moreover, induced subgraphs correspond to induced subcomplexes:
\[
\Delta(G)|_W =  \Delta(G|_W)\qquad (W\subseteq V).
\]

\subsection{Graded Betti numbers, regularity, and Hochster's formula}

For a finitely generated graded $R$-module $M$, we write $\beta_{i,j}(M)$ for its graded Betti numbers,
and set
\begin{align*}
\pdim M &:= \max\{i:\beta_{i,j}(M)\neq 0 \text{ for some } j\},\\
\reg M &:= \max\{j-i:\beta_{i,j}(M)\neq 0\}.
\end{align*}
We repeatedly use the following standard bounds coming from short exact sequences.

\begin{lem}\label{lem:ses2}
Given a short exact sequence of finitely generated graded $R$-modules
\[
0 \longrightarrow A \longrightarrow B \longrightarrow C \longrightarrow 0,
\]
one has
\begin{enumerate}
\item[(a)] $\pdim B \le \max\{\pdim A,\pdim C\}$,
\item[(b)] $\pdim A \le \max\{\pdim B,\pdim C-1\}$,
\item[(c)] $\pdim C \le \max\{\pdim A+1,\pdim B\}$,
\item[(d)] $\reg B \le \max\{\reg A,\reg C\}$,
\item[(e)] $\reg A \le \max\{\reg B,\reg C+1\}$,
\item[(f)] $\reg C \le \max\{\reg A-1,\reg B\}$.
\end{enumerate}
\end{lem}

\begin{lem}\label{lem:ses}
Let $J\subseteq R$ be a homogeneous ideal and let $z$ be a variable. The short exact sequence
\[
0 \longrightarrow R/(J:z)(-1)\xlongrightarrow{\ \cdot z\ } R/J \longrightarrow R/(J,z) \longrightarrow 0
\]
implies
\[
\reg R/J \le \max\{\reg R/(J,z),\ \reg R/(J:z)+1\}.
\]
Moreover, if $\reg R/(J,z)\neq \reg R/(J:z)+1$, then equality holds with the larger of the two.
\end{lem}

\begin{remark}{(Hochster's Formula)}\label{rem:Hochster}
Let $\Delta$ be a simplicial complex on vertex set $V=\{1,\ldots,n\}$, and let
$R=\Bbbk[x_1,\ldots,x_n]$. Then
\[
\beta_{i,j}(R/I_{\Delta})=\sum_{\substack{W\subseteq V\\ |W|=j}}
\dim_{\Bbbk}\widetilde H_{j-i-1}(\Delta|_W;\Bbbk),
\]
where $\Delta|_W$ denotes the induced subcomplex on $W$.
\end{remark}

\begin{remark}\label{rem:reg_pdim}
For a finite graph $G$, Hochster's formula and $I_{\Delta(G)}=I(G)$ give
\begin{enumerate}
\item[(a)] $
\reg R/I(G)=1+\max\{r:\widetilde H_r(\Delta(G|_W);\Bbbk)\neq 0\text{ for some }W\subseteq V(G)\},
$
\item[(b)] $
\pdim R/I(G)=\max\{|W|-1-r:\widetilde H_r(\Delta(G|_W);\Bbbk)\neq 0\text{ for some }W\subseteq V(G)\}.
$
\end{enumerate}
\end{remark}

\begin{remark}\label{rem:bight_pdim}
    It follows from \cite[Proposition 4.7.]{ProjDim} that 
    $$\pdim R/I(G)\geq \bight I(G)$$
\end{remark}

If $G_1$ and $G_2$ are graphs, we denote by $G_1\sqcup G_2$ their disjoint union and by $K_2$ a single
edge.

\begin{lem}\label{lem:ind-join}
For finite graphs $G_1$ and $G_2$ there is a natural isomorphism
\[
  \Delta(G_1\sqcup G_2)\ \cong\ \Delta(G_1)*\Delta(G_2),
\]
where $*$ denotes the simplicial join.
In particular, $\Delta(K_2)\cong S^0$, and if $G$ is a disjoint union of $k$ edges then
\[
  \Delta(G)\ \cong\ (S^0)^{*k}\ \simeq\ S^{k-1}.
\]
\end{lem}

\begin{proof}
An independent set in $G_1\sqcup G_2$ is uniquely determined by its restrictions to $G_1$ and~$G_2$,
and each restriction is an independent set. This gives the isomorphism with the join, see for
example \cite{JonssonGraphs, KozlovInd}. The statements for $K_2$ and a disjoint union of edges follow
by induction, together with the fact that $(S^0)^{*k}$ is a sphere of dimension $k-1$.
\end{proof}

\begin{remark}
Note that $S^{k-1}$ can be viewed as the boundary of a cross-polytope.
\end{remark}

We also make use of the following well-known result from the literature.

\begin{lem}\label{lem:disjoint_sum}
Let $G$ be a finite simple graph. If $G$ can be written as a disjoint union of graphs
$G_1,\ldots,G_s$ (on pairwise disjoint vertex sets), then
\[
\reg R/I(G) =  \sum_{i=1}^s \reg R_i/I(G_i),
\]
where $R_i=\Bbbk[V(G_i)]$.
\end{lem}


\begin{lem}\label{lem:partial-cone-complex}
Let $H$ be a graph and fix a subset $C\subseteq V(H)$. Let $G$ be the graph obtained from $H$ by
adjoining a new vertex $z$ and connecting $z$ to every vertex in $C$ (and to no other vertices).
Then
\begin{align*}
  \Delta(G) =  \Delta(H)\ \cup\ \bigl(z * \Delta(H|_{V(H)\setminus C})\bigr),\\
\Delta(H)\ \cap\ \bigl(z * \Delta(H|_{V(H)\setminus C})\bigr) =  \Delta(H|_{V(H)\setminus C}).  
\end{align*}

\end{lem}

\begin{proof}
An independent set in $G$ either does not contain $z$ (hence is an independent set of $H$), or
contains $z$, in which case it may only use vertices in $V(H)\setminus C$. This gives the stated
union description, and the intersection consists of independent sets in $H$ that avoid $C$.
\end{proof}

\subsection{Discrete Morse matchings}

In this subsection we recall the basic language of discrete Morse theory from \cite{Forman98} and record the specific
consequences we need for independence complexes of paths and cycles.

\begin{defn}[Discrete Morse matching]
Let $X$ be a finite simplicial complex and let $\mathcal{F}(X)$ be its face poset ordered by
inclusion. The Hasse diagram of $\mathcal{F}(X)$ has a vertex for each face and an edge between
$\sigma\subset\tau$ whenever $\dim(\tau)=\dim(\sigma)+1$.

A \emph{matching} $M$ on $\mathcal{F}(X)$ is a set of edges in the Hasse diagram such that each face
of $X$ is incident to at most one edge in $M$. We orient each edge in the Hasse diagram as follows:
\begin{itemize}
  \item if the edge is \emph{unmatched}, we orient it upward, from a face to its coface;
  \item if the edge is \emph{matched}, say $(\sigma,\tau)\in M$ with $\sigma\subset\tau$, we orient
  it downward, from $\tau$ to $\sigma$.
\end{itemize}
The matching $M$ is called \emph{acyclic} (or a \emph{Morse matching}) if the resulting
directed graph contains no directed cycles. A face of $X$ that is incident to no edge of $M$ is
called a \emph{critical} face (or \emph{critical cell}).
\end{defn}

\begin{thm}{\cite[Theorems 7.3 and 8.2]{Forman98}}
\label{thm:forman}
Let $X$ be a finite simplicial complex equipped with an acyclic matching $M$ on its face poset.
Then there exists a CW complex $Y$ that is homotopy equivalent to $X$ with exactly one $p$-cell for
each critical $p$-simplex of $M$.

Moreover, the homology of $X$ is computed by a chain complex $C^{M}_\ast(X)$ whose $p$-th chain
group is the free $\Bbbk$-vector space on the critical $p$-simplices, and whose boundary maps are
defined by signed counts of gradient paths between critical cells.
\end{thm}

We  also use that acyclic matchings restrict to subcomplexes.

\begin{lem}
\label{lem:restriction-matching}
Let $Y\subseteq X$ be a subcomplex and let $M$ be an acyclic matching on $X$. Consider the matching
$M|_Y$ obtained by restricting $M$ to those matched pairs whose two faces both lie in $Y$. Then
$M|_Y$ is an acyclic matching on~$Y$.

In particular, every face of $Y$ that is critical for $M$ is also critical for $M|_Y$.
\end{lem}

\begin{proof}
The restriction $M|_Y$ is clearly a matching on the face poset of $Y$. Any directed cycle in the
oriented Hasse diagram of $Y$ would also be a directed cycle in the oriented Hasse diagram of $X$,
contradicting the acyclicity of $M$. Thus $M|_Y$ is acyclic.
\end{proof}

\subsection{Independence polynomials of graphs}

\begin{defn}\label{def:indpoly}
The \emph{independence polynomial} of a graph $G$ is
\[
P_G(x) := \sum_{i=0}^{\alpha(G)} g_i x^i,
\]
where $g_i$ is the number of independent sets of cardinality $i$ in $G$. (In particular, $g_0=1$
counts the empty independent set.)
\end{defn}

\begin{remark}
The polynomial $P_G(x)$ from \Cref{def:indpoly} appears in the literature under other names.
Prodinger--Tichy in \cite{ProdingerTichy1982} define the \emph{Fibonacci number} of a graph to be the
total number of independent sets; in our notation this is $P_G(1)$. Building on this work, Hopkins--Staton in \cite{Fibonacci} define the \emph{Fibonacci polynomial} of $G$ and prove that
it expands as $\sum_{k\ge 0} F_k(G)x^k$, where $F_k(G)$ counts $k$-element independent sets of $G$
\cite[Proposition~4]{Fibonacci}. Comparing coefficients shows that their Fibonacci polynomial is
exactly our independence polynomial $P_G(x)$ (with $F_k(G)=g_k$). Finally, the term \emph{independence
polynomial} is commonly attributed to the work of Gutman--Harary  in \cite{GutmanHarary1983}.
\end{remark}

\begin{remark}
Some common independence polynomials that we reference throughout the paper include:
\begin{itemize}
\item the independence polynomial of a single vertex: $1+x$;
\item the independence polynomial of $K_2$ (a single edge): $P_{K_2}(x)=1+2x$.
\end{itemize}
\end{remark}

\begin{defn}
Let $v\in V$. The \emph{open neighborhood} of $v$ is
\[
N(v):=\{u\in V : \{u,v\}\in E\},
\]
and the \emph{closed neighborhood} is $N[v]:=N(v)\cup\{v\}$.
\end{defn}

The following is a standard recurrence for independence polynomials.

\begin{lem}\label{lem:delete_contract}
Let $G$ be a finite simple graph and let $v\in V(G)$. Then
\[
P_G(x) =  P_{G-v}(x)\ +\ x P_{G-N[v]}(x).
\]
\end{lem}

Recall that we denote the path on $n$ vertices by $P_n$ and the cycle on $n$ vertices by $C_n$.

\begin{remark}\label{rem:ind_poly_path_cycle}
Applying \Cref{lem:delete_contract} to paths and cycles yields the recurrences
\begin{enumerate}
\item[(a)] $P_{P_n}(x)=P_{P_{n-1}}(x)+x P_{P_{n-2}}(x)$,
\item[(b)] $P_{C_n}(x)=P_{P_{n-1}}(x)+x P_{P_{n-3}}(x)$.
\end{enumerate}
\end{remark}

The following closed forms are obtained by solving the linear recurrence in
\Cref{rem:ind_poly_path_cycle}(a) and substituting
into \Cref{rem:ind_poly_path_cycle}(b).

\begin{lem}\label{rem:closed_ind_path_cycle}
The independence polynomials of $P_n$ and $C_n$ admit the closed forms
\begin{align*}
P_{P_n}(x)
&=\frac{1}{2^{n+1} s}\Big((1+2x+s)(1+s)^n+(s-1-2x)(1-s)^n\Big),\\
P_{C_n}(x)
&=\frac{1}{2^{n}}\Big((1+s)^n+(1-s)^n\Big),
\end{align*}
where $s=\sqrt{1+4x}$.
\end{lem}

The next invariant of $P_G(x)$ (or namely, $G$) was introduced in \cite{biermann2026realizableregdeghpairs} in the
context of the degree of the $h$-polynomial. It will be useful in our $\mathfrak a$-invariant computations.

\begin{defn}\label{def:MG}
Let $G$ be a graph with independence polynomial $P_G(x)$. Define
\[
M(G) := \ord_{x=-1} P_G(x),
\]
the multiplicity of $x=-1$ as a root of $P_G(x)$.
\end{defn}

\begin{remark}\label{rem:deg-h}
It was shown in \cite[Theorem 4.4]{biermann2026realizableregdeghpairs} that
\[
\deg h_{R/I(G)}(t)=  \alpha(G)-M(G).
\]
Moreover, it also follows from \cite[Theorem 3.2]{biermann2024degree} and \cite[Remark 2.10]{biermann2026realizableregdeghpairs} that
$M(G)=0$ if and only if $P_G(-1)\neq 0$.
\end{remark}

We  use the following Hilbert-series convention for the $\mathfrak a$-invariant: for a standard
graded $\Bbbk$-algebra $A$ of dimension $d$, write its Hilbert series as
\[
\Hilb_A(t) =  \frac{h_A(t)}{(1-t)^d},
\]
where $h_A(t)$ is a polynomial, and set $\mathfrak a(A):=\deg h_A(t)-d$. 

\begin{remark}\label{rem:a_M}
Recall from \cite[Remark 4.5]{biermann2026realizableregdeghpairs} that $\mathfrak a(R/I(G)) =  -M(G)$.
\end{remark}

\section{Suspension operations from the literature}\label{sec:coning-operations}

In this section we recall two closely related suspension constructions on graphs from the literature (see \cite{herzog2011monomial,IndMatch}). After summarizing the relevant known results, we prove our main result of this section: suspension over a \emph{minimal} vertex cover increases the projective dimension by one. 

Let $G=(V,E)$ be a graph and let $\mathcal C\subseteq V$. We write $G(\mathcal C)$ for the graph obtained from $G$ by adjoining a new vertex $z$ and joining $z$ to every vertex in $\mathcal C$ (and to no other vertices). The two suspensions of interest are:
\begin{itemize}
\item the \emph{full suspension} $\widehat{G}=G(V)$, in which $z$ is adjacent to every vertex of $G$;
\item the \emph{cover suspension} $G(\mathcal C)$, where $\mathcal C$ is a minimal vertex cover of $G$.
\end{itemize}
Throughout, we take $R=\Bbbk[x_1,\dots,x_n]$ with $V=\{x_1,\dots,x_n\}$ and set $S:=R[z]$.

\subsection{Full suspension}

The edge ideal of the full suspension $ \widehat{G}$ of $G$ is
\[
I(\widehat{G}) \;=\; I(G) + (zx_i : x_i\in V)\ \subseteq\ S.
\]
We first recall the following statement  relating the graded Betti numbers of $I(G)$ and $I(\widehat{G})$ from the literature  (see \cite[Lemma 3.1]{MousivandProductGraphs}).

\begin{thm}\label{thm:full-cone-betti}
For all $i,j\ge 0$ we have:
\begin{enumerate}
\item[(1)] If $j>i+1$, then
\[
\beta_{i,j}\bigl(S/I(\widehat{G})\bigr)
\;=\;
\beta_{i,j}\bigl(R/I(G)\bigr)+\beta_{i-1,j-1}\bigl(R/I(G)\bigr).
\]
\item[(2)] If $j=i+1$ and $i\ge 1$, then
\[
\beta_{i,i+1}\bigl(S/I(\widehat{G})\bigr)
\;=\;
\beta_{i,i+1}\bigl(R/I(G)\bigr)+\beta_{i-1,i}\bigl(R/I(G)\bigr)+\binom{n}{i},
\]
where $n=|V(G)|$. In particular, $\beta_{0,1}(S/I(\widehat{G}))=0$.
\end{enumerate}
\end{thm}



Thus one has the following result relating the regularity and projective dimension of $I(\widehat{G})$ to that of $I(G)$. 
\begin{cor}\label{cor:full-cone-reg-pdim}
We have $\reg S/I(\widehat{G})=\reg R/I(G)$ and  $\pdim S/I(\widehat{G}) = n$.
\end{cor}


One can also relate the Hilbert series of $I(G)$ and $I(\widehat{G})$.

\begin{remark}\label{rem:full-cone-hilb}
As noted above, $\Delta(\widehat{G})=\Delta(G)\cup\{z\}$ with disjoint vertex sets. For simplicial complexes $\Delta$, $\Delta'$
on disjoint vertex sets one has
\[
\Hilb_{\Bbbk[\Delta\cup\Delta']}(t)=\Hilb_{\Bbbk[\Delta]}(t)+\Hilb_{\Bbbk[\Delta']}(t)-1
\]
(see, e.g., \cite[Proposition~4.9]{MousivandProductGraphs}).
Applying this with $\Delta'=\{z\}$ gives
\[
\Hilb_{S/I(\widehat{G})}(t)=\Hilb_{R/I(G)}(t)+\frac{t}{1-t} =\frac{h_{R/I(G)}(t)+t(1-t)^{d-1}}{(1-t)^d}
\]
where $\Hilb_{R/I(G)}(t)=\dfrac{h_{R/I(G)}(t)}{(1-t)^d}$ with $d=\dim(R/I(G))=\alpha(G)$.
\end{remark}

Lastly, we discuss the relationship between the $\mathfrak a$-invariants  of $I(G)$ and $I(\widehat{G})$. Recall that $\mathfrak{a}(R/I(G))\leq 0$ for any finite simple graph $G$.

\begin{cor}\label{cor:ainv-compare-full-cone}
We have
\[
\mathfrak{a}(S/I(\widehat{G})) \;=\; \deg\!\big(h_{R/I(G)}(t)+t(1-t)^{d-1}\big)\;-\;d.
\]
In particular,
\[
\mathfrak{a}(S/I(\widehat{G}))=
\begin{cases}
0, & \mathfrak{a}(R/I(G))<0,\\[2pt]
0 \text{ or a negative integer}, & \mathfrak{a}(R/I(G))=0.
\end{cases}
\]
Moreover, if $\mathfrak{a}(R/I(G))=0$ and the leading coefficient of $h_{R/I(G)}(t)$ satisfies
\(
[t^d] h_{R/I(G)}(t)\neq (-1)^d,
\)
then $\mathfrak{a}(S/I(\widehat{G}))=0$.
\end{cor}

\begin{proof}
This is an immediate degree comparison using \Cref{rem:full-cone-hilb} and $\dim(S/I(\widehat{G}))=d$.
Since $t(1-t)^{d-1}$ has degree $d$, the degree of the numerator is $\deg h_{R/I(G)}(t)$ if $\deg h_{R/I(G)}(t)>d$
and is $d$ if $\deg(h_G)<d$. When $\deg h_{R/I(G)}(t)=d$, cancellation in the $t^d$--term occurs exactly
when $[t^d]h_{R/I(G)}(t)=(-1)^d$, forcing the degree below $d$ and hence $\mathfrak{a}(S/I(\widehat{G}))<0$.
\end{proof}

\subsection{Suspension over minimal vertex covers}

Let $G$ be a finite simple graph on $n$ vertices with no isolated vertices, and let $\mathcal C$ be a
minimal vertex cover of $G$.  Then
\[
I(G(\mathcal C)) \;=\; I(G) + (zx_i : x_i\in \mathcal C)\ \subseteq\ S.
\]
Write $U:=V(G)\setminus \mathcal C$ and $u:=|U|$. Since $\mathcal C$ is a vertex cover, $U$ is
independent.

We note that the suspension $G(\mathcal C)$, where $\mathcal C$ is a vertex cover, is a special case of the $S$--suspension of \cite{IndMatch} (take $S=V(G)\setminus \mathcal C$). Consequently, the regularity preservation in part~(a) holds for \emph{every} vertex cover $\mathcal C$ by \cite[Lemma~1.5]{IndMatch}, and we do not repeat its proof here. By contrast, to the best of our knowledge the corresponding statement on projective dimension in part~(b) has not previously appeared in the literature. We therefore include a proof of~(b), which relies essentially on the minimality of $\mathcal C$. This dependence on minimality also motivates our later study of suspensions over independent sets, which we view as a complementary direction to the vertex cover ($S$-suspension) framework.

\begin{thm}\label{thm:cover-cone-reg-pdim}
We have
\begin{enumerate}
\item[(a)] $\reg S/I(G(\mathcal C)) = \reg R/I(G)$,
\item[(b)] $\pdim S/I(G(\mathcal C)) = \pdim R/I(G)+1$.
\end{enumerate}
\end{thm}

\begin{proof}
Consider the standard short exact sequence
\[
0 \longrightarrow S/(I(G(\mathcal C)):z)(-1) \xlongrightarrow{\ \cdot z\ } S/I(\widehat{G})
\longrightarrow S/(I(\widehat{G}),z) \longrightarrow 0.
\]
Since $z$ does not appear in $I(G)$, we have $I(G):z=I(G)$, and
\[
(I(G(\mathcal C)):z) \;=\; I(G) + (x_i : x_i\in \mathcal C).
\]
Because $\mathcal C$ is a vertex cover, every generator $x_ix_j$ of $I(G)$ is divisible by some
$x_i\in \mathcal C$. So $I(G)\subseteq (x_i:x_i\in \mathcal C)$ and hence
\[
(I(G(\mathcal C)):z)=(x_i:x_i\in \mathcal C).
\]
Moreover, we have $(I(G(\mathcal C)),z)=(I(G),z)$.
\begin{enumerate}
\item[(a)] This result follows from \cite[Lemma~1.5]{IndMatch}.
\item[(b)] Notice that
\begin{align*}
\pdim S/ (I(G),z) &=\pdim R/I(G)+1,\\
\pdim S/ (I(G(\mathcal C)):z) &= \pdim R/ (x_i : x_i\in \mathcal C) = |\mathcal C|.
\end{align*}
Because $\mathcal C$ is a minimal vertex cover, we have $|\mathcal C|\le \bight I(G) \le \pdim R/I(G)$  by \Cref{rem:bight_pdim}.  Then \Cref{lem:ses2}(a) gives
\[
\pdim S/ I(G(\mathcal C)) \leq \max \{\pdim R/I(G)+1 , |\mathcal C|\}
= \pdim R/I(G)+1.
\]
For the reverse inequality, set $p=\pdim R/I(G)$. Let $\Delta=\Delta(G)$ and
$\Delta'=\Delta(G(\mathcal C))$. By \Cref{rem:reg_pdim}(b), there exists $W\subseteq V(G)$ and $r\ge 0$
such that $\widetilde H_r(\Delta|_W)\neq 0$ and $p=|W|-r-1$. Set $W':=W\cup\{z\}$ and
$K:=U\cap W$, where $U=V(G)\setminus \mathcal C$.

Faces of $\Delta'|_{W'}$ not containing $z$ are exactly the faces of $\Delta|_W$.
Faces containing $z$ are precisely $\{z\}\cup F$ with $F\subseteq K$. So
\[
\Delta'|_{W'}=\Delta|_W  \cup  (z*2^{K})
\qquad\text{and}\qquad
\Delta|_W\cap(z*2^{K})=2^{K}
\]
where $2^K$ is the full simplex on $K$.
Assume first that $K\neq\emptyset$. Since $2^{K}$ is a simplex and $z*2^{K}$ is a cone, both are
contractible. Hence $\widetilde H_i(2^{K})=\widetilde H_i(z*2^{K})=0$ for all $i\ge 0$.
Since the outer terms in the following Mayer--Vietoris sequence vanish
\[
\cdots \longrightarrow
\widetilde H_r(2^{K})
\longrightarrow
\widetilde H_r(\Delta|_W)\oplus \widetilde H_r(z*2^{K})
\longrightarrow
\widetilde H_r(\Delta'|_{W'})
\longrightarrow
\widetilde H_{r-1}(2^{K})
\longrightarrow \cdots 
\]
we have
$\widetilde H_r(\Delta'|_{W'})\cong \widetilde H_r(\Delta|_W)\neq 0$.

If $K=\emptyset$, then the only face containing $z$ is $\{z\}$ and 
$\Delta'|_{W'}=\Delta|_W\sqcup\{z\}$. In this case reduced homology in degrees $\ge 1$ is unchanged,
and for $r=0$ the dimension of $\widetilde H_0$ increases by $1$. In particular, nonvanishing is
preserved. Thus $\widetilde H_r(\Delta'|_{W'})\neq 0$ in all cases. It then follows from \Cref{rem:reg_pdim}(b) that
\[
\pdim S/I(\widehat{G})\ \ge\ |W'|-r-1= p+1=\pdim R/I(G)+1.
\]
Hence $\pdim S/I(\widehat{G})=\pdim R/I(G)+1$. \qedhere
\end{enumerate}
\end{proof}

\begin{prop}\label{prop:ainv-min-vertex-cover}
Recall that $U=V(G)\setminus \mathcal C$ and $u=|U|$.

(a) We have
\[
M(G(\mathcal C))=
\begin{cases}
M(G), & M(G)<u,\\
u, & M(G)>u,\\
\ge u, & M(G)=u.
\end{cases}
\]
Equivalently, $M(G(\mathcal C))=\min\{M(G),u\}$ unless $M(G)=u$, in which case cancellation may increase
the multiplicity. Moreover $M(G(\mathcal C))>u$ holds if and only if $M(G)=u$ and the constant term $q(0)$
below equals $1$.

\smallskip
(b) If $\mathfrak{a}(R/I(G))>-u$ then $\mathfrak{a}(S/I(\widehat{G}))=\mathfrak{a}(R/I(G))$,
while if $\mathfrak{a}(R/I(G))<-u$ then $\mathfrak{a}(S/I(\widehat{G}))=-u$.
\end{prop}

\begin{proof}
Since $U$ is independent, an independent set of $G(\mathcal C)$ either does not contain $z$ (hence is an independent
set of $G$), or contains $z$, in which case it can be completed by an arbitrary subset of $U$. Hence
\[
P_{\widehat{G}}(x)=P_G(x)+x(1+x)^u.
\]
Write $y:=1+x$ and expand $P_G(x)$ near $y=0$. Then $P_G(x)=y^m q(y)$ with $m=M(G)$ and $q(0)\neq 0$, and 
\[
P_{G(\mathcal C)}(x)=y^{m}q(y)+(y-1)y^{u}.
\]
If $m<u$, then $P_{G(\mathcal C)}(x)=y^{m}Q(y)$ where $Q(y)=q(y)+y^{u-m}(y-1)$ and $Q(0)=q(0)\neq 0$. So $M(G(\mathcal C))=m$.
If $m>u$, then $P_{G(\mathcal C)}(x)=y^{u}Q'(y)$ where $Q'(y)=y^{m-u}q(y)+y-1$ and $Q'(0)=-1\neq 0$. So $M(G(\mathcal C))=u$.
If $m=u$, then $P_{G(\mathcal C)}(x)=y^{u}\bigl(q(y)+y-1\bigr)$ has order at least $u$, with strict inequality exactly when
$q(0)=1$. This proves (a).

Statement of (b) follows from  $\mathfrak{a}(R/I(H))=-M(H)$ for $H=G,~G(\mathcal C)$ by \Cref{rem:a_M} and combining it with (a).
\end{proof}

\section{Suspension over maximal independent sets:  cycles}\label{sec:cycles}

In this section we restrict to cycles and investigate their suspensions over maximal independent sets. We denote a cycle on $n$ vertices by $C_n$ and start with the case $n\equiv 0 \pmod 3$, and moreover suspend over a specific class of maximal independent sets. This regime requires a different set of arguments than the general case, and thus merits a separate treatment. The next subsection develops the theory for cycles in full generality.

\subsection{A special suspension case: cycles with wide spokes}\label{subsec:wide_spokes}
In this subsection, we work with $C_n$ when $n\equiv 0 \pmod 3$. We are particularly interested in  suspension of  $C_n$ 
over one of its maximal independent sets with the smallest size  $\lceil \frac{n}{3} \rceil$. Notice that such maximal independent sets are precisely  the following ones:
\begin{align*}
    \mathcal{C}_1 &= \{x_1,x_4,\ldots, x_{3k-2}\},\\
   \mathcal{C}_2  &= \{x_2,x_5,\ldots, x_{3k-1}\},\\
    \mathcal{C}_3  &= \{x_3,x_6,\ldots, x_{3k}\}.
\end{align*}

We follow the notation introduced above for the rest of this subsection. Notice that   any $\mathcal{C}_i$-suspension of  $C_n$ for $i=1,2,3$ are equivalent by symmetry. So, we  focus on $\mathcal{C}_1$-suspension of $C_n$ for the remainder of this section.

\begin{defn}
        Let $G:=C_n(\mathcal{C}_1)$  be $\mathcal{C}_1$-suspension of  $C_n$ such that
 $$I(G)=I(C_n)+ (zx_i : x_i\in \mathcal{C}_1)\subseteq S=R[z].$$ 
 When $n\equiv 0 \pmod 3$, we call the graph $G$ the \emph{$n$ cycle with wide spokes}.
\end{defn}

\begin{figure}[!htb]
    \centering
    \begin{minipage}{.5\textwidth}
        \centering
        \includegraphics[width=0.65\linewidth]{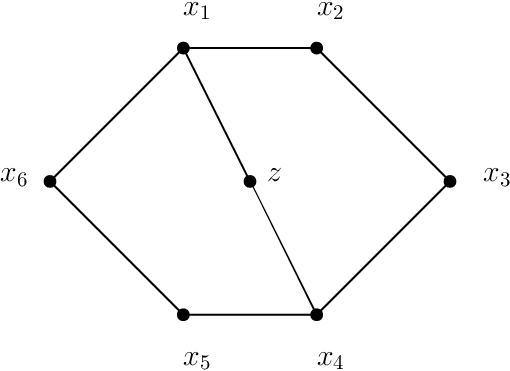}
        \caption{6 cycle with wide spokes}
        \label{fig:6cycle}
    \end{minipage}%
    \begin{minipage}{0.5\textwidth}
        \centering
        \includegraphics[width=0.5\linewidth]{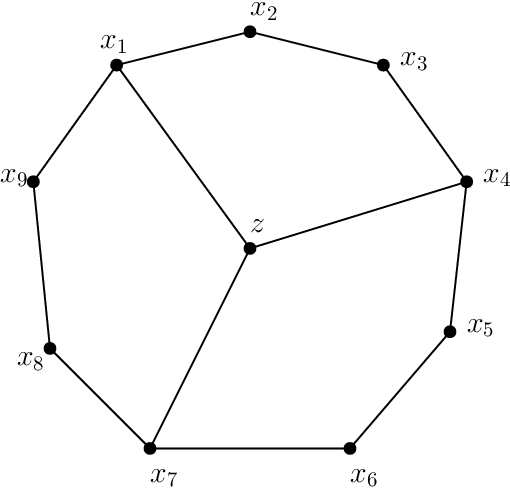}
        \caption{9 cycle with wide spokes}
        \label{fig:9cycle}
    \end{minipage}
\end{figure}

Our goal in this section is to show that regularity is preserved when one passes from $C_n$ to $G$. Before we do this, we first set-up our notation, then we recall and establish some useful lemmas needed for the proof.

\begin{notation}
\label{setup:cycle}
Let $H$ denote the induced subgraph of $C_{3k}$ on $W := V(C_{3k})\setminus \mathcal{C}_1$. The edges of $H$ are $\{x_{3i-1},x_{3i}\}$ for $i=1,\ldots k$. So, $H$ is a disjoint union of $k$ edges.

Consider the independence complexes
\[
  \Delta := \Delta(C_{3k}), \qquad A := \Delta\big(H\big),
\]
and the inclusion of simplicial complexes $i\colon A\hookrightarrow \Delta$.

Let
\[
  \sigma_1 :=  \{2,5,8,\dots,3k-1\},\qquad
  \sigma_2 :=  \{3,6,\dots,3k\}.
\]
Notice  $\sigma_1$ and $\sigma_2$ are facets of $A\subset \Delta$
 since   $\mathcal{C}_2$ and $\mathcal{C}_3$ are maximal independent sets in $H$ and $C_{3k}$.  
 \end{notation}

For the independence complex of the cycle, we use a discrete Morse description due to Kozlov from \cite{KozlovBook,KozlovInd}.

\begin{lem}
\label{lem:Delta-critcells}
In the situation of Setup~\ref{setup:cycle}, we have
\[\Delta  \simeq S^{k-1} \vee S^{k-1}.\]
Moreover, there exists an acyclic matching $M_\Delta$ on the face poset of $\Delta$ such that:
\begin{enumerate}
  \item $M_\Delta$ has no critical simplices of dimension $\ge k$; and
  \item the only critical $(k-1)$-simplices are exactly $\sigma_1$ and $\sigma_2$.
\end{enumerate}
Consequently $\widetilde H_{k-1}(\Delta;\Bbbk)\cong \Bbbk^2$, with basis represented by the classes of $\sigma_1$ and $\sigma_2$.
\end{lem}

\begin{proof}
A discrete Morse matching for $\Delta(C_n)$ is constructed in \cite[Proposition 5.2]{KozlovInd} and \cite[Proposition 11.17]{KozlovBook}. It follows from Kozlov's work that, for $n=3k$, $\Delta(C_n)$  has exactly two critical $(k-1)$-simplices and no critical simplices in higher dimension. It follows that $\Delta(C_{3k})$ is homotopy equivalent to a wedge of two $(k-1)$-spheres as proved in \cite{KozlovBook, KozlovInd}.

By rotating the labeling of the cycle if necessary, the two critical cells produced by Kozlov's matching are precisely the $(k-1)$-faces $\sigma_1$ and $\sigma_2$ from Setup~\ref{setup:cycle}. Denote the corresponding Morse matching by $M_\Delta$. Let $C^{M_\Delta}_\ast(\Delta)$ be the Morse complex associated to $M_\Delta$.  Since  $\widetilde H_{k-1}(\Delta;\Bbbk) \simeq \Bbbk^2$,  Theorem~\ref{thm:forman}  implies that
\[
  \widetilde H_{k-1}(\Delta;\Bbbk)\ \cong\ C^{M_\Delta}_{k-1}(\Delta)
  \ \cong\ \Bbbk\cdot\sigma_1 \ \oplus\ \Bbbk\cdot\sigma_2,
\]
and that the homology classes of $\sigma_1$ and $\sigma_2$ form a basis of $\widetilde H_{k-1}(\Delta;\Bbbk)$.
\end{proof}

We now explain how the inclusion $i\colon A\hookrightarrow\Delta$ interacts with these matchings and the resulting top homology.

\begin{prop}
\label{prop:inclusion-nonzero}
In the situation of Setup~\ref{setup:cycle}, the inclusion
\[
  i\colon A\hookrightarrow\Delta
\]
induces a nonzero homomorphism
\[
  i_{\star}: \widetilde H_{k-1}(A;\Bbbk) \longrightarrow \widetilde H_{k-1}(\Delta;\Bbbk).
\]
In particular, since $\dim_\Bbbk \widetilde H_{k-1}(A;\Bbbk)=1$, the map $i_{\star}$ is injective.
\end{prop}

\begin{proof}
Restrict the matching $M_\Delta$ from Lemma~\ref{lem:Delta-critcells} to the subcomplex $A$ to obtain a matching $M_A$ on $A$. By Lemma~\ref{lem:restriction-matching}, $M_A$ is an acyclic matching on $A$. The faces $\sigma_1$ and $\sigma_2$ lie in $A$ and are unmatched in $M_\Delta$. Hence, they are also critical $(k-1)$-simplices for $M_A$. In addition, note that $M_A$ has no critical simplices of dimension $\geq k$ since $M_{\Delta}$ does not have any such critical cells.

Let $C^{M_A}_\ast(A)$ and $C^{M_\Delta}_\ast(\Delta)$ be the Morse complexes associated to $M_A$ and $M_\Delta$, respectively. In degree $k-1$, we have
\[
  C^{M_A}_{k-1}(A) \;\cong\; \Bbbk\cdot\sigma_1 \;\oplus\; \Bbbk\cdot\sigma_2,
\]
because these are exactly the critical $(k-1)$--faces for $M_A$. The Morse differential
\[
  \partial_{k-1}^{M_A}\colon C^{M_A}_{k-1}(A) \longrightarrow C^{M_A}_{k-2}(A)
\]
is some $\Bbbk$--linear map determined by gradient paths. Since $\widetilde H_{k-1}(A;\Bbbk)\cong \Bbbk$ by \Cref{lem:ind-join} and $C^{M_A}_{k}(A)=0$, we have
\[
  \dim_\Bbbk \ker \partial_{k-1}^{M_A} = 1.
\]
So, there exist $a,b\in\Bbbk$ with $(a,b)\neq (0,0)$ such that
\[
  \partial_{k-1}^{M_A}(a\sigma_1 + b\sigma_2) = 0.
\]
Thus, $\ker \partial_{k-1}^{M_A}$ is spanned by $a\sigma_1 + b\sigma_2$. The corresponding homology class
\[
  [\omega] := [ a\sigma_1 + b\sigma_2 ] \in \widetilde H_{k-1}(A;\Bbbk)
\]
is a generator of $\widetilde H_{k-1}(A;\Bbbk)$.

Since $M_\Delta$ has no critical simplices in dimension $\ge k$, we have
$C^{M_\Delta}_{k}( \Delta)=0$. Moreover,
$C^{M_\Delta}_{k-1}(\Delta)\cong \Bbbk\cdot\sigma_1\oplus \Bbbk\cdot\sigma_2$
and $\widetilde H_{k-1}(\Delta;\Bbbk)\cong \Bbbk^2$ by Lemma~\ref{lem:Delta-critcells}.
Therefore $\partial^{M_\Delta}_{k-1}=0$, and the classes of $\sigma_1$ and $\sigma_2$
form a basis of $\widetilde H_{k-1}(\Delta;\Bbbk)$.

The inclusion $i\colon A\hookrightarrow\Delta$ induces a chain map between Morse complexes
\[
  i_{\star}\colon C^{M_A}_\ast(A) \longrightarrow C^{M_\Delta}_\ast(\Delta),
\]
which in degree $k-1$ is  the inclusion
\[
  i_{\star}\colon \Bbbk\cdot\sigma_1 \oplus \Bbbk\cdot\sigma_2
  \hookrightarrow 
  \Bbbk\cdot\sigma_1 \oplus \Bbbk\cdot\sigma_2,
\]
sending $\sigma_j$ to the same simplex viewed in $\Delta$, for $j=1,2$. So,  we obtain
\[
  i_{\star}([\omega]) 
  \;=\; [ a\sigma_1 + b\sigma_2 ] 
  \;\in\; \widetilde H_{k-1}(\Delta;\Bbbk)
  \;\cong\; \Bbbk^2.
\]
Since $[\sigma_1],[\sigma_2]$ is a basis of $\widetilde H_{k-1}(\Delta;\Bbbk)$ and $(a,b)\neq (0,0)$, we have $i_{\star}([\omega])\neq 0$. Therefore, $i_{\star}$ is a nonzero  map on $\widetilde H_{k-1}$ and it is injective.
\end{proof}

\begin{thm}\label{thm:cycle_spokes}
Let $G$ be the $n$ cycle with wide spokes. Then 
$$\reg S/ I(G)= \reg R/I(C_n).$$
\end{thm}

\begin{proof}
It follows from  \cite[Theorem 5.2]{Selvi_cycles} that  $\reg R/I(C_{3k}) = k$. Since $C_{3k}$ is an induced subgraph of $G$, we have $ k=\reg R/I(C_{3k})  \leq \reg S/I(G)$. It remains to prove $\reg S/I(G) \leq k$.

    For this upper bound on $\reg S/I(G)$, we use Remark~\ref{rem:reg_pdim} (a):
    $$ \reg S/I(G) = 1+ \max \{ r : \widetilde{H}_r (\Delta(G|_W)) \neq 0 \text{ for some } W \subseteq V(G)\}.$$
Our goal is to show that, for any $W \subseteq V(G)$, one has  $\widetilde{H}_r (\Delta(G|_W))=0$ for $r\geq k$. 

    Let $\Delta^0= \Delta (G)$ so that $\Delta^0|_W= \Delta(G|_W)$. Following the notation in Setup~\ref{setup:cycle} and \Cref{lem:partial-cone-complex}, we have
    $$\Delta^0= \Delta \cup c(A)$$
    where  $c(A)$ is the cone $z\star A$  with apex $z$ such that $\Delta \cap c(A)=A$.

    If $z\notin W$, then $W\subseteq V(C_n)$ and $G|_W=C_n|_W$. So, $\widetilde{H}_r (\Delta(G|_W))  \cong  \widetilde{H}_r (\Delta(C_n|_W))=0$ for $r\geq k$ and  for any $W\subseteq V(C_n)$ since $\reg R/I(C_n)=k$.

    If $z\in W$, let $W=T\cup \{z\}$. Then
    $$\Delta^0|_W = \Delta|_T \cup c( A|_T)$$
    where $\Delta|_T= \Delta (C_n|_T)$ and $A_T= \Delta (H|_{T \cap V(H)}).$

    The following observations will be useful in the following part of the proof:
    \begin{itemize}
        \item Since  $C_n|_T$ is an induced subgraph of $C_n$, we have  $\reg R/ I(C_n|_T) \leq \reg R/ I(C_n)= k$. This  implies that $\widetilde{H}_r (\Delta|_T) =0$ for $r\geq k$.
        \item We have $A_T = \Delta (H')$ where  $H'=H|_{T \cap V(H)}$. Then $\reg R/I(H')\leq \reg R/I(H) \leq \reg R/I(C_n) = k$ since $H'$ is  an induced subgraph of $H$ (which itself is an induced subgraph of $C_n$). So, we get $\widetilde{H}_r (A|_T) =0$ for $r\geq k$.
    \end{itemize}
Using these two observations in the following Mayer-Vietoris sequence
$$  ~~~~~~\cdots \longrightarrow \widetilde{H}_r (A|_T) \xlongrightarrow {i_{\star}} \widetilde{H}_r (\Delta|_T)  \longrightarrow \widetilde{H}_r (\Delta^0|_W) \longrightarrow \widetilde{H}_{r-1} (A|_T)  \longrightarrow  \cdots$$
we conclude that $\widetilde{H}_r (\Delta^0|_W) =0$ for $r>k$.

If $r=k$, we can update the above long exact sequence as
$$  ~~~~~~ 0 \longrightarrow \widetilde{H}_k (\Delta^0|_W) \longrightarrow \widetilde{H}_{k-1} (A|_T) \longrightarrow  \widetilde{H}_{k-1} (\Delta|_T)   \longrightarrow  \widetilde{H}_{k-1} (\Delta^0|_W)  \longrightarrow   \cdots$$

Observe that if $T \cap V(H) \subset V(H)$, then $\reg R/I(H') \leq k-1$. This means  $\widetilde{H}_{k-1} (A|_T)=0$ which implies that $\widetilde{H}_k (\Delta^0|_W)=0$.  For the remainder of the proof, suppose 
$$T \cap V(H) = V(H).$$ 
Notice that $A|_T=A$ in this case. Since $V(H) \subseteq T\subseteq V(C_n)$, we consider the following three cases to complete the proof:

\textbf{Case 1.} If $T=V(C_n)$, then $\Delta^0|_W=\Delta^0$ and $\Delta|_T=\Delta$. We have
    $$
    0 \longrightarrow \widetilde{H}_k (\Delta^0) \longrightarrow \widetilde{H}_{k-1} (A) \xlongrightarrow {i_{\star}}  \widetilde{H}_{k-1} (\Delta)   \longrightarrow \cdots.
    $$
Since $i_{\star}$ is injective by \Cref{prop:inclusion-nonzero}, we conclude that $\widetilde{H}_k (\Delta^0|_W)=0$ in this case.

\textbf{Case 2.} If $T=V(H)$, then $\Delta|_T=A|_T= A$ and
     $$
     0 \longrightarrow \widetilde{H}_k (\Delta^0) \longrightarrow \widetilde{H}_{k-1} (A) \xlongrightarrow {i_{\star}}  \widetilde{H}_{k-1} (A)   \longrightarrow \cdots.
    $$
This means that $i_{\star}$ is the identity map, thus injective. So, $\widetilde{H}_k (\Delta^0|_W)=0$ in this case as well.

\textbf{Case 3.} If $V(H) \subset T\subset V(C_n)$, we have  $A \subset \Delta|_T \subset \Delta$ where
\[
   j \colon A \hookrightarrow \Delta|_T \text{ and } k \colon \Delta|_T \hookrightarrow\Delta  \text{ such that }   i=k\circ j \colon A \hookrightarrow \Delta
\]
induces the homomorphism
\[
  j_{\star}: \widetilde H_{k-1}(A;\Bbbk) \longrightarrow \widetilde H_{k-1}(\Delta|_T;\Bbbk) \text{ and }  k_{\star}: \widetilde H_{k-1}(\Delta|_T;\Bbbk) \longrightarrow \widetilde H_{k-1}(\Delta;\Bbbk)
\]
such that  $k_{\star} \circ  j_{\star} = i_{\star}$. So, the updated Mayer-Vietoris sequence is    
    $$  ~~~~~~ 0 \longrightarrow \widetilde{H}_k (\Delta^0|_W) \longrightarrow \widetilde{H}_{k-1} (A|_T) \xlongrightarrow{j_{\star}}  \widetilde{H}_{k-1} (\Delta|_T)   \longrightarrow    \cdots$$
  Since $i_{\star}$ is injective by \Cref{prop:inclusion-nonzero}, we conclude that $j_{\star}$ is injective which implies that $\widetilde{H}_k (\Delta^0|_W)=0$, as desired. 
\end{proof}

\subsection{Suspension over any maximal independent set}

Let   $\mathcal{C}$  be a maximal independent set of $C_n$ and let $G:=C_n(\mathcal{C})$ be the $\mathcal{C}$-suspension of  $C_n$  such that
$$I(G)= I(C_n)+(zx_i: x_i\in \mathcal{C})  \subseteq S=R[z].$$
 Our goal in this subsection is to detect how the regularity, projective dimension and the $\mathfrak a$-invariant changes from $C_n$ to $G$. It is worth pointing out that no assumptions on $n$ are made, unlike in the previous subsection.

\begin{observation}\label{obs:cycle_delete_C}
Let $ H:=G-N[z]$.  Then  \(H\) is an induced
subgraph of \(C_n\) on \(V(C_n)\setminus \mathcal C\) since  \(z\) is adjacent exactly to the vertices of \(\mathcal C\). Let $\ell:=n-2|\mathcal{C}|$. Because $\mathcal{C}$ is a maximal independent set of a cycle, the induced
subgraph on $V(C_n)\setminus \mathcal{C}$ has no path of length $\ge 2$. Hence it is a disjoint union of
$\ell$ edges and \(|\mathcal C|-\ell\) isolated vertices. 
\end{observation}

\begin{thm}\label{thm:cycle_reg_pdim}
Suspensions of cycles over a maximal independent set preserves regularity and increases projective dimension by one:
 \begin{enumerate}
     \item[(a)] $\reg S/ I(G) = \reg R/I(C_n)$. 
     \item[(b)] $\pdim S/ I(G) = \pdim R/I(C_n)+1$.
 \end{enumerate}
\end{thm}

\begin{proof}
Consider the short exact sequence
\[
0 \longrightarrow S/(I(G):z)(-1)\xrightarrow{\cdot z} S/I(G)
\longrightarrow S/(I(G),z)\longrightarrow 0.
\]
Since $z$ does not appear in $I(C_n)$, we have $(I(G),z)=(I(C_n),z)$ and
\begin{align*}
  \reg S/(I(G),z) &=\reg R/I(C_n),\\ 
\pdim S/(I(G),z)&=\pdim R/I(C_n)+1.  
\end{align*}
Moreover,
\[
I(G):z = I(H)+(x_i:x_i\in\mathcal{C}).
\]

Recall from \Cref{obs:cycle_delete_C} that  $I(H)$ is generated by $\ell=n-2|\mathcal{C}|$
pairwise coprime quadrics, and together with the $|\mathcal{C}|$ variables $(x_i:x_i\in\mathcal{C})$
these form a complete intersection with $\ell+|\mathcal{C}|=n-|\mathcal{C}|$ minimal generators. Therefore
\[
\reg S/(I(G):z)=\ell=n-2|\mathcal{C}|
\qquad\text{and}\qquad
\pdim S/(I(G):z)=n-|\mathcal{C}|.
\]

\textup{(a)} Notice that
$\reg R/I(C_n) \le \reg S/I(G)$ since $C_n$ is an induced subgraph of $G$. For the reverse inequality, combining conclusions from above and  \Cref{lem:ses2}(d) yields to
\[
\reg S/I(G)\le  \max\{\reg R/I(C_n),\ n-2|\mathcal{C}|+1\}.
\]
 Hence it suffices to show
\[
n-2|\mathcal{C}|+1 \le \reg R/I(C_n).
\]
Using the fact that $|\mathcal{C}|\ge \lceil n/3\rceil$, we have
\[
n-2|\mathcal{C}|+1 \le n-2\lceil n/3\rceil+1=
\begin{cases}
k+1 &\text{if }n=3k+2,\\
k   &\text{if }n=3k+1,\\
k+1 &\text{if }n=3k.
\end{cases}
\]
By \cite[Theorem~5.2]{Selvi_cycles},
\[
\reg R/I(C_n)=
\begin{cases}
k+1 &\text{if }n=3k+2,\\
k   &\text{if }n=3k+1,\\
k   &\text{if }n=3k.
\end{cases}
\]
Thus the desired inequality holds for $n\equiv 1,2\pmod 3$. When $n=3k$, it holds as soon as
$|\mathcal{C}|\ge k+1$ (since then $n-2|\mathcal{C}|+1\le k$). The remaining case
$n=3k$ and $|\mathcal{C}|=k$ is exactly the situation handled in \Cref{thm:cycle_spokes}.
Therefore $\reg S/I(G)=\reg R/I(C_n)$.

\medskip
\noindent\textup{(b)}
By \Cref{lem:ses2}(a), we have
\[
\pdim S/I(G)\le \max\{\pdim S/(I(G):z),\ \pdim S/(I(G),z)\}.
\]
It follows from the initial conclusions of the proof that 
$$
\max\{\pdim S/(I(G):z),\ \pdim S/(I(G),z)\}
=\max\{n-|\mathcal{C}|,\ \pdim R/I(C_n)+1\}.
$$ 
Let $\mathcal{M}:=V(C_n)\setminus \mathcal{C}$. Since $\mathcal{C}$ is maximal independent set, $\mathcal{M}$
is a minimal vertex cover of $C_n$. Hence $\pdim R/I(C_n)\ge \bight I(C_n)\ge |\mathcal{M}|=n-|\mathcal{C}|$ by
\Cref{rem:bight_pdim}. So the maximum above equals
$\pdim R/I(C_n)+1$. Therefore
\[
\pdim S/I(G)\le \pdim R/I(C_n)+1.
\]
For the reverse inequality, \Cref{lem:ses2}(c) gives
\[
\pdim R/I(C_n)+1\le \max\{\pdim S/I(G),\ n-|\mathcal{C}|+1\}.
\]
It  suffices to show 
$$
\pdim S/I(G)\ge n-|\mathcal{C}|+1.
$$
Observe that $\mathcal{M}\cup\{z\}$ is a minimal
vertex cover of $G$. To see this, notice that removing $z$ leaves an edge $zx_i$ uncovered for any $x_i\in\mathcal{C}$, while
removing $v\in\mathcal{M}$ leaves uncovered an edge $vw$ of the cycle with $w\in\mathcal{C}$ (such a
neighbor exists by maximality of $\mathcal{C}$). Hence, by \Cref{rem:bight_pdim}, we have 
$\pdim S/I(G)\ge  \bight I(G)\ge |\mathcal{M}|+1=n-|\mathcal{C}|+1$. This completes the proof.
\end{proof}

The following result about the $\mathfrak a$-invariant of cycles is a consequence of \cite[Theorem 4.5]{biermann2024degree}. Even though its proof is immediate from \cite{biermann2024degree}, we include the proof to show the usefulness of the multiplicity $M(G)$ approach. 

\begin{cor}\label{cor:only_cycle_a-inv}
$\mathfrak{a} (R/I(C_n))=0$.
\end{cor}

\begin{proof}
Recall from \Cref{rem:a_M} that \(\mathfrak{a}(R/I(C_n))=-M(C_n)\).
Hence it suffices to prove $M(C_n)=0$ which is equivalent to  \(P_{C_n}(-1)\neq 0\) by \Cref{rem:deg-h}.

We have the following recursion from \Cref{rem:ind_poly_path_cycle}:
\[
P_{C_n}(x)=P_{P_{n-1}}(x)+xP_{P_{n-3}}(x).
\]
Set \(a_m:=P_{P_m}(-1)\). Then
\[
P_{C_n}(-1)=a_{n-1}-a_{n-3}.
\]
From the path recurrence \(P_{P_m}(x)=P_{P_{m-1}}(x)+xP_{P_{m-2}}(x)\) from \Cref{rem:ind_poly_path_cycle} with
\(P_{P_0}(x)=1\) and \(P_{P_1}(x)=1+x\), we obtain
\[
a_0=1,\qquad a_1=0,\qquad a_m=a_{m-1}-a_{m-2}\ \ (m\ge 2).
\]
Computing the first few terms gives
\[
(a_m)_{m\ge 0} = 1, 0, -1, -1, 0, 1, 1, 0, -1, -1, 0, 1,\ldots.
\]
So \((a_m)\) is \(6\)-periodic. A quick verification shows
\[
a_{n-1}-a_{n-3}\in\{2,1,-1,-2,-1,1\}.
\]
Hence \(P_{C_n}(-1)\neq 0\). Thus $M(C_n)=0$.
\end{proof}

\begin{thm}\label{thm:cycle_a-inv}
The $\mathfrak a$-invariant is preserved for cycles under suspension over maximal independent set.  In particular, we have $\mathfrak{a} (S/ I(G)) = \mathfrak{a} (R/I(C_n))=0$.
\end{thm}

\begin{proof}
It suffices to prove $M(G)=M(C_n)=0$ by \Cref{rem:a_M}. Since  $M(C_n)=0$  from \Cref{cor:only_cycle_a-inv}, we focus on proving $M(G)=0$ which is equivalent to 
\(P_G(-1)\neq 0\).

Since \(G-z=C_n\) and $H=G-N[z]$ from \Cref{obs:cycle_delete_C}, applying the recursion from \Cref{lem:delete_contract}  to $G$ at \(z\) results with
\[
P_G(x)=P_{C_n}(x)+x P_{H}(x).
\]
Recall from  \Cref{obs:cycle_delete_C} that $H$ has   \(\ell\) edges and  \(|\mathcal C|-\ell\) isolated vertices of \(H\) where \(\ell=n-2|\mathcal{C}|\). Then, by multiplicativity of independence polynomials over disjoint unions, we have
\[
P_H(x)=(1+x)^{|\mathcal C|-\ell}(1+2x)^\ell.
\]
Substituting this into the above recursion gives
\begin{equation}\label{eq:PG-cycle-cone}
P_G(x)=P_{C_n}(x)+x(1+x)^{|\mathcal C|-\ell}(1+2x)^\ell.
\end{equation}

We now evaluate at \(x=-1\).

\textbf{Case 1:} Suppose \(|\mathcal C|>\ell\).
Then \((1+x)^{|\mathcal C|-\ell}\) vanishes at \(x=-1\). So \eqref{eq:PG-cycle-cone} yields
\[
P_G(-1)=P_{C_n}(-1).
\]
Since $M(C_n)=0$ by \Cref{cor:only_cycle_a-inv}, we have $P_{C_n}(-1)\neq 0$. Thus $M(G)=M(C_n)=0$.

\textbf{Case 2:} Suppose \(|\mathcal C|=\ell\).
Then \(H\) has no isolated vertices and it is a disjoint union of \(\ell\) edges. This means \(|V(H)|=2\ell\)
and  \(n=3\ell\).
Evaluating \eqref{eq:PG-cycle-cone} at \(x=-1\) gives
\[
P_G(-1)=P_{C_{3\ell}}(-1)+(-1)^{\ell+1}.
\]
From the same \(6\)-periodic computation in the proof of \Cref{cor:only_cycle_a-inv},
\(P_{C_{3\ell}}(-1)=2\) when \(3\ell\equiv 0\pmod 6\) and \(P_{C_{3\ell}}(-1)=-2\) when
\(3\ell\equiv 3\pmod 6\); equivalently,
\[
P_{C_{3\ell}}(-1)=2(-1)^\ell \neq 0.
\]
Thus
\[
P_G(-1)=2(-1)^\ell+(-1)^{\ell+1}=(-1)^\ell\neq 0.
\]
Therefore, \Cref{rem:a_M}, we have the desired outcome:
 \[
 \mathfrak{a}(S/I(G))=\mathfrak{a}(R/I(C_n)=M(G)=M(C_n)=0.\qedhere
 \]
\end{proof}

\section{Suspension over maximal independent sets:  paths}\label{sec:paths}

In this final section we turn to paths and study their suspensions over maximal independent sets. As in the preceding section on cycles, our goal is to understand how this suspension operation affects the three algebraic invariants of interest: $\reg$, $\pdim$, and the $\mathfrak a$--invariant.

We begin by fixing notation. Let $P_n$ denote the path on vertices $x_1,\ldots,x_n$ with edge ideal
\[
I(P_n)=(x_1x_2,\;x_2x_3,\;\ldots,\;x_{n-1}x_n)\subseteq R=\Bbbk[x_1,\ldots,x_n].
\]
Let $\mathcal{C}$ be a maximal independent set of $P_n$ and let $G:=P_n(\mathcal{C})$ be the $\mathcal{C}$-suspension of $P_n$. Equivalently,
\[
I(G)= I(P_n)+\bigl(zx_i:\; x_i\in \mathcal{C}\bigr)\subseteq S=R[z].
\]
Finally, set $P:=G\setminus N[z]$. Since $N[z]=\{z\}\cup \mathcal{C}$, the graph $P$ is precisely the induced subgraph of $P_n$ on $V(P_n)\setminus \mathcal{C}$.

\begin{lem}\label{lem:number-edges-P}
Let  $\ell$ denotes the number of edges of $P$. Then
\[
\ell \;\le\; \left\lfloor \frac{n-1}{3} \right\rfloor.
\]
\end{lem}

\begin{proof}
Define a $0/1$-string $(\gamma_1,\dots,\gamma_n)$ by
\[
\gamma_i =
\begin{cases}
1 & \text{if } x_i \in \mathcal{C},\\[2pt]
0 & \text{if } x_i \notin \mathcal{C}.
\end{cases}
\]
Since $\mathcal{C}$  is independent, we have $\gamma_i \gamma_{i+1} \neq 11$ for all $i$.  
Since $\mathcal{C}$  is maximal, every vertex not in $\mathcal{C}$  has a neighbor in $\mathcal{C}$. This is equivalent to
\[
\gamma_1 = 0 \Rightarrow \gamma_2 = 1,\quad
\gamma_n = 0 \Rightarrow \gamma_{n-1} = 1,\quad
\gamma_i = 0 \Rightarrow \gamma_{i-1}=1 \text{ or } \gamma_{i+1}=1
\ (2 \le i \le n-1).
\]
In particular, the substring $000$ never occurs, and we cannot have 
$\gamma_1=\gamma_2=0$ or $\gamma_{n-1}=\gamma_n=0$.  
Thus any pair of consecutive zeros $\gamma_i=\gamma_{i+1}=0$ satisfies $2 \le i \le n-2$.

An edge $x_i x_{i+1}$ lies in $P$ if and only if $\gamma_i=\gamma_{i+1}=0$.  
Let $x_i x_{i+1} \in E(P)$. 
Then for $x_i$ (which is not in $\mathcal{C}$), maximality forces $\gamma_{i-1}=1$ (since $\gamma_{i+1}=0$),  
and for $x_{i+1}$, maximality forces $\gamma_{i+2}=1$ (since $\gamma_i=0$).  
This means every edge of $P$ sits inside a block
\[
\gamma_{i-1}\gamma_i\gamma_{i+1}\gamma_{i+2} = 1 0 0 1.
\]

Let \(E(P)=  \{ \{x_{i_1}, x_{i_1+1}\}, \ldots, \{x_{i_{\ell}}, x_{i_{\ell}+1}\}\}\)
where $i_1<\dots<i_\ell$. Then
\begin{equation}\label{eq:spacing}
i_{r+1} - i_r \;\ge\; 3 \quad \text{for all } r = 1,\dots,\ell-1.
\end{equation}

By \eqref{eq:spacing} we have $i_r \ge 2 + 3(r-1) = 3r-1$ for all $r$, and in particular
\[
3\ell - 1 \;\le\; i_\ell \;\le\; n-2.
\]
Hence $3\ell \le n-1$, so $\ell \le (n-1)/3$. Since $\ell$ is an integer, we have
\(\ell \;\le\; \left\lfloor \frac{n-1}{3} \right\rfloor\), as claimed.
\end{proof}

We adopt the notation introduced in Lemma~\ref{lem:number-edges-P} for the rest of this section.

\begin{remark}\label{rem:extremal-C}
The bound in Lemma~\ref{lem:number-edges-P} is sharp for all $n$.   Recall that each edge of $P$ comes from a pair of consecutive vertices of $\mathcal{C}$ 
at distance~$3$, so $\ell$ is the number of such gaps. Then, for each $n$ there exists a maximal independent set $\mathcal{C}$  with
\[
|E(P)| = \ell = \left\lfloor \frac{n-1}{3} \right\rfloor,
\]
and for $n\equiv 1 \pmod 3$ this extremal configuration is unique.

\begin{enumerate}
    \item[(a)] If $n=3k$, take
\[
\mathcal{C} = \{x_1,x_4,\dots,x_{3k-2},x_{3k}\},
\]
which has gaps of sizes $3,\dots,3,2$, hence $k-1=\lfloor (n-1)/3\rfloor$ gaps of size~$3$.

    \item[(b)] If $n=3k+1$, take
\[
\mathcal{C}= \{x_1,x_4,\dots,x_{3k+1}\},
\]
which has $k=\lfloor (n-1)/3\rfloor$ gaps of size~$3$. Moreover, if a maximal independent set $\mathcal{C}$ 
satisfies $\ell=k$, then all gaps between consecutive vertices of $\mathcal{C}$  must have size~$3$, which forces
$\mathcal{C}=\{x_1,x_4,\dots,x_{3k+1}\}$. In all other cases $\mathcal{C}$  has at least one gap of size~$2$, and hence
$\ell<k$.

    \item[(c)] If $n=3k+2$, again set
\[
\mathcal{C} = \{x_1,x_4,\dots,x_{3k+1}\},
\]
which has $k=\lfloor (n-1)/3\rfloor$ gaps of size~$3$ and is maximal since every vertex outside $\mathcal{C}$ 
is adjacent to one of these vertices.
\end{enumerate} 
\end{remark}

\begin{thm}\label{thm:path_reg}
Suspension of a path over a maximal independent set preserves regularity
$$\reg S/ I(G)= \reg R/I(P_n)$$ 
except when $n\equiv 1 \pmod 3$ and $\mathcal{C}=\{x_1,x_4,\ldots, x_{3k+1}\}$. In this case, we have 
$$\reg S/ I(G)= \reg R/I(P_n)+1.$$
\end{thm}

\begin{proof}
    Consider the following short exact sequence
$$0 \longrightarrow S/(I(G):z) (-1) \xlongrightarrow{ \cdot z} S/I(G) \longrightarrow S/(I(G),z)  \longrightarrow  0.$$
Observe that  
\begin{align*}
 (I(G),z) & =  (I(P_n),z), \text{ and } \\
 I(G):z & =  I(P) + (x_i : x_i \in \mathcal{C})
\end{align*}

 It follows from \Cref{lem:ses2} (d) that
 $$\reg S/ I(G) \leq \max\{ \reg R/I(P_n), \reg R/I(P)+1 \}.$$
As we investigate this upper bound and the maximum of the right hand-side, we recall that   $\ell \leq \lfloor \frac{n-1}{3} \rfloor$ from \Cref{lem:number-edges-P} and $\reg R/I(P)=\ell$. In addition, we keep in mind that $ \reg R/I(P_n) \leq \reg S/I(G)$ since $P_n$ is an induced subgraph of $G$ and $\reg R/I(P_n)= \lfloor \frac{n+1}{3} \rfloor$ (see \cite[Theorem 4.7]{Selvi_cycles}).  Putting these together results with the following cases:
 
 \textbf{Case 1.} If $n=3k$, then $\reg S/I(G)=  k= \reg R/I(P_n)$, as desired.

 \textbf{Case 2.}  If $n=3k+2$, then $\reg S/I(G)=  k+1= \reg R/I(P_n)$, as desired.

 \textbf{Case 3.}  If $n=3k+1$, then  $k=\reg R/I(P_n)\leq \reg S/I(G) \leq \max\{\ell+1, k\}$.  Recall from Remark~\ref{rem:extremal-C} that $\ell=k$ only when $\mathcal{C}=\{x_1,x_4,\ldots, x_{3k+1}\}$. So,  when $\mathcal{C} \neq \{x_1,x_4,\ldots, x_{3k+1}\} $, we have  $\ell \leq k-1$ and
    $$\reg S/I(G)= k =\reg R/I(P_n).$$
Suppose $\mathcal{C} =\{x_1,x_4,\ldots, x_{3k+1}\} $.  Then  $\reg S/I(G)\leq k+1$ from the above discussion. It suffices to show $\reg S/I(G)\geq k+1$ which is equivalent to showing $\widetilde{H}_k (\Delta(G|_W)) \neq 0$ for some $W\subseteq V(G)$ from Remark~\ref{rem:reg_pdim} (a).

    Let $\Delta^0= \Delta (G), ~~\Delta=\Delta(P_n)$ and $A=\Delta(P)$. It then follows from \Cref{lem:partial-cone-complex}  that $\Delta^0= \Delta \cup c(A)$ where $\Delta \cap c(A)=A$. It is shown in \cite[Proposition 4.6]{KozlovInd} that $\Delta$ is contractible. Then, by the following Mayer-Vietoris sequence
     $$\cdots   \longrightarrow \widetilde{H}_{r} (\Delta)  \longrightarrow \widetilde{H}_r (\Delta^0) \longrightarrow \widetilde{H}_{r-1} (A) \longrightarrow \widetilde{H}_{r-1} (\Delta) \longrightarrow \cdots
    $$
   we have $\widetilde{H}_r (\Delta^0) \cong \widetilde{H}_{r-1} (A) $ for all $r$.

Since $P$ has $k$ disjoint edges, it follows from \Cref{lem:ind-join} that 
\(  A \ \simeq\ S^{k-1}\) and the only non-zero homology of $A$ occurs at degree $(k-1)$:
$$\widetilde H_{k}(\Delta^0;\Bbbk)\cong \widetilde H_{k-1}(A;\Bbbk)\cong  \Bbbk.$$ 
Then  $\reg S/I(G)\geq k+1$ since $\widetilde{H}_k (\Delta(G|_W)) \neq 0$ for $W=V(G)$. Hence, 
$$\reg S/I(G)= \reg R/I(P_n)= k+1$$
 in the exceptional case. This completes the proof.
\end{proof}

The following result will be useful as we investigate the projective dimension of $S/I(G)$. 

\begin{prop}\label{prop:ell-plus-C}
Let  $\mathcal{C}$  be a maximal independent set of $P_n$ and $ |\mathcal{C}|=t$. Then
\[
\ell + t \;\le\; \left\lfloor \frac{2n+1}{3} \right\rfloor
\]
where $\ell = |E(P)|$. Equivalently, if $n=3k$, then $\ell+t\le 2k$, and if $n=3k+1$ or $n=3k+2$, then
$\ell+t \le 2k+1$. Moreover, when $n=3k+1$ and $\mathcal{C} \neq \{x_1,x_4,\dots,x_{3k+1}\}$,
we have the sharper bound
\[
\ell + t \;\le\; \left\lceil \frac{2(n-1)}{3} \right\rceil = 2k.
\]
\end{prop}

\begin{proof}
As in the proof of \Cref{lem:number-edges-P}, record $\mathcal{C}$ by a $0/1$ string $\gamma=(\gamma_1,\dots,\gamma_n)$ where $\gamma_i=1$ iff $x_i\in \mathcal{C}$.
Independence gives that no substring $11$ occurs. Maximality gives that no substring $000$ occurs, and
also $\gamma$ cannot begin or end with $00$ (otherwise an endpoint would have no neighbor in $\mathcal{C}$).

Consequently, between any two consecutive $1$'s there are either one or two $0$'s.  Equivalently, as we
move from one vertex of $\mathcal{C}$ to the next, the gap is either $01$ or $001$. Let $p$ be the number of gaps
of type $01$ and let $q$ be the number of gaps of type $001$ between the first and last $1$. Let
$\delta_0,\delta_t\in\{0,1\}$ record whether there is an initial $0$ before the first $1$ and a terminal $0$
after the last $1$, respectively, and set $\delta=\delta_0+\delta_t\in\{0,1,2\}$. Then the number of $1$'s is
\[
t=| \mathcal{C}|=p+q+1,
\]
and the segment from the first to the last $1$ has length $1+2p+3q$. Including the possible endpoint
$0$'s gives
\[
n=\delta+(1+2p+3q).
\]
Moreover, a gap of type $001$ contributes exactly one pair of consecutive $0$'s, hence exactly one edge
in the induced subgraph $P=P_n\setminus  \mathcal{C}$. So
\[
\ell=|E(P)|=q.
\]
Therefore
\[
\ell+t = q+(p+q+1)=p+2q+1.
\]

Now maximize $p+2q+1$ subject to $n=\delta+1+2p+3q$ with $p,q\ge 0$ and $\delta\in\{0,1,2\}$.
From $2p=n-\delta-1-3q\ge 0$ we have $q\le \frac{n-\delta-1}{3}$, and using $p=\frac{n-\delta-1-3q}{2}$,
\[
\ell+t=p+2q+1=\frac{n+1-\delta+q}{2}.
\]
Thus $\ell+t$ is increasing in $q$ and decreasing in $\delta$, so it is maximized when $\delta=0$ and $q$
is as large as possible. Then
\[
\ell+t \le
\begin{cases}
2k & \text{if } n=3k,\\
2k+1 & \text{if } n=3k+1,\\
2k+1 & \text{if } n=3k+2,
\end{cases}
\]
which is equivalent to $\ell+t\le \big\lfloor\frac{2n+1}{3}\big\rfloor$.

When $n=3k+1$, the value $2k+1$ occurs exactly when $\delta=0$ and $p=0$, i.e.\ all gaps are of type
$001$, so $ \mathcal{C}=\{x_1,x_4,\dots,x_{3k+1}\}$. For any other maximal independent set, either $\delta\ge 1$ or
$p\ge 1$, and in either case $q\le k-1$, which forces
\[
\ell+t=\frac{n+1-\delta+q}{2}\le \frac{(3k+1)+1+(k-1)}{2}=2k.
\]
This gives the sharper bound claimed in the proposition.
\end{proof}

Now, we are ready to discuss projective dimension.

\begin{thm}\label{thm:path_pdim}
Suspension of a path over a maximal independent set increases the projective dimension by one:
$$\pdim S/ I(G)= \pdim R/I(P_n) +1.$$
\end{thm}

\begin{proof}
Let $d:=\pdim R/I(P_n)$. Recall that $d=\left\lceil\frac{2(n-1)}{3}\right\rceil$ by \cite[4.1.2 Proposition]{Bouchat}.
Consider the short exact sequence
\[
0 \longrightarrow S/(I(G):z)(-1)\xrightarrow{\cdot z} S/I(G)\longrightarrow S/(I(G),z)\longrightarrow 0.
\]
As in \Cref{thm:path_reg}, we have $(I(G),z)=(I(P_n),z)$ and $I(G):z=I(P)+(x_i:x_i\in \mathcal{C})$, where
$P$ is the induced subgraph of $P_n$ on $V(P_n)\setminus \mathcal{C}$ and $\ell=|E(P)|$. Since $P$ is a disjoint union
of $\ell$ edges and isolated vertices, the ideal $I(P)+(x_i:x_i\in \mathcal{C})$ is a complete intersection on
$\ell$ quadrics and $|\mathcal{C}|$ linear forms. Hence
\begin{align*}
   \pdim S/(I(G),z)&=\pdim R/I(P_n)+1=d+1,\\
\pdim S/(I(G):z)&=\ell+|\mathcal{C}|. 
\end{align*}
It then follows from \Cref{lem:ses2} that
\[
\pdim S/I(G)\le \max\{\ell+|\mathcal{C}|,\ d+1\}.
\]
Using \Cref{prop:ell-plus-C} results with $\ell+|\mathcal{C}|\le \left\lfloor\frac{2n+1}{3}\right\rfloor\le d+1$. So $\pdim S/I(G)\le d+1$.

For the reverse inequality, we produce a minimal vertex cover of $G$ with $d+1$ vertices.
First note that any maximal independent set $\mathcal{C}$ of the path $P_n$ contains $x_1$ or $x_2$.
Define an independent set $\mathcal{D}\subseteq V(P_n)$ by
\[
\mathcal{D}=
\begin{cases}
\{x_1,x_4,x_7,\dots\} & \text{if } x_1\in \mathcal{C},\\
\{x_2,x_5,x_8,\dots\}\cup\{x_n\}\ \text{(only if }n\equiv 1\!\!\pmod 3) & \text{if } x_1\notin \mathcal{C}.
\end{cases}
\]
One can verify that $\mathcal{D}$ is an independent dominating set of $P_n$, hence a maximal independent set,
and $|\mathcal{D}|=\lceil n/3\rceil$. By construction $\mathcal{D}\cap \mathcal{C}\neq\emptyset$. So $z$ is adjacent to a vertex of $\mathcal{D}$. Therefore, $\mathcal{D}$ is also a maximal independent set in $G$. Thus $V(G)\setminus \mathcal{D}$ is a minimal vertex cover of $G$ of size
\[
|V(G)\setminus \mathcal{D}|=(n+1)-\left\lceil\frac{n}{3}\right\rceil
=
\left\lceil\frac{2(n-1)}{3}\right\rceil+1
=d+1.
\]
Consequently $\pdim S/I(G) \ge \bight I(G)\ge d+1$  by \Cref{rem:bight_pdim}.
Combining with the upper bound yields $\pdim S/I(G)=\pdim R/I(P_n)+1$, as desired.
\end{proof}

Before we relate the $\mathfrak a$-invariants of $P_n$ and $G$, we first determine this invariant for $P_n$. As in the cycle case, this is a consequence of \cite[Theorem 4.5]{biermann2024degree} but we provide a proof using the multiplicity of $-1$ as a root of the independence polynomial for completeness.

\begin{thm}\label{cor:only_path_a_inv}
$
\mathfrak a (R/I(P_n))= 
\begin{cases}
 0, & n\equiv 0,2\pmod 3,\\
 -1, & n\equiv 1\pmod 3.
\end{cases}
$
\end{thm}

\begin{proof}
By  \Cref{rem:a_M}, it suffices to compute $M(P_n)$. 

As in the proof of \Cref{cor:only_cycle_a-inv}, set \(a_m:=P_{P_m} (-1)\) and recall that \((a_m)\) is \(6\)-periodic with values \(1,0,-1,-1,0,1,\ldots\). In particular,
\[
a_m\neq 0 \ \text{if } m\equiv 0,2 \pmod 3,
\qquad
a_m=0 \ \text{if } m\equiv 1 \pmod 3.
\]
Hence \(M(P_n)=0\) if \(n\equiv 0,2\pmod 3\). When \(n\equiv 1\pmod 3\), it remains to show that the root at \(-1\) is simple, i.e. $M(P_n)=1$.   Let \(b_m:=P_{P_m}'(-1)\). Differentiating \(P_{P_m} (x)=P_{P_{m-1}}(x)+x P_{P_{m-2}}(x)\) and evaluating at \(-1\) gives
\begin{equation}\label{eq:bn-recur}
b_m=b_{m-1}+a_{m-2}-b_{m-2}\qquad (m\ge 2),
\end{equation}
with \(b_0=0\) and \(b_1=1\). Taking \(m=3k+4\) in \eqref{eq:bn-recur} and using \eqref{eq:bn-recur} once more at
\(m=3k+3\) (together with \(a_{3k+1}=0\)) yields
\[
b_{3k+4}=- b_{3k+1}+a_{3k+2}.
\]
From the \(6\)-periodicity, \(a_{3k+2}=(-1)^{k+1}\). Thus the above relation
becomes
\begin{equation}\label{eq:b-3k+4}
b_{3k+4}=- b_{3k+1}+(-1)^{k+1}\qquad (k\ge 0).
\end{equation}
We claim that
\[
b_{3k+1}=(-1)^k(k+1)\qquad (k\ge 0).
\]
For \(k=0\), this is \(b_1=P_{P_1}'(-1)=1\). Assume it holds for some \(k\). Since \(3(k+1)+1=3k+4\),
\eqref{eq:b-3k+4} gives
\[
b_{3k+4}=- b_{3k+1}+(-1)^{k+1}
=- \bigl((-1)^k(k+1)\bigr)+(-1)^{k+1}=(-1)^{k+1}(k+2),
\]
which is the desired formula for \(k+1\). This completes the induction. Therefore
\[
P_{P_{3k+1}}'(-1)=b_{3k+1}=(-1)^k(k+1)\neq 0.
\]
So \(M(P_n)=1\) when \(n\equiv 1\pmod 3\). This completes the proof.
\end{proof}

\begin{thm}\label{thm:path_a_inv}
Suspension of a path over a maximal independent set preserves the $\mathfrak a$-invariant
 $$\mathfrak{a} (S/ I(G))= \mathfrak{a} (R/I(P_n))$$  
 except when $n\equiv 1 \pmod 3$ and $\mathcal{C}=\{x_1,x_4,\ldots, x_{3k+1}\}$. In this case, we have 
 $$\mathfrak{a} (S/ I(G))= \mathfrak{a} (R/I(P_n))+1.$$  
\end{thm}

\begin{proof}
By Remark \ref{rem:a_M}, the theorem is equivalent to the following statement:
\[
M(G)=M(P_n)
\]
except when $n\equiv 1\pmod 3\text{ and }\mathcal C=\{x_1,x_4,\ldots,x_{3k+1}\}$. In this case \(M(G)=M(P_n)-1\). 

Since we computed $M(P_n)$ (hence $\mathfrak a (R/I(P_n))$) in \Cref{cor:only_path_a_inv}, we focus on  \(M(G)\).

Adapt the notation from the proof of  \Cref{prop:ell-plus-C} where we encode \(\mathcal C\) by the \(0/1\)-string
\(\gamma_1\cdots\gamma_n\) with \(\gamma_i=1\) if and only if \(x_i\in\mathcal C\). Recall that \(P=G-N[z]\) is a disjoint union of \(\ell\) edges and \(e\) isolated vertices where
\[
e=|\mathcal C|-\ell-1+\delta.
\]
It then follows from \Cref{lem:delete_contract} and multiplicativity of independence polynomials over disjoint unions that
\begin{equation}\label{eq:PG-path}
P_G(x)=P_{P_n}(x)+x(1+x)^e(1+2x)^\ell.
\end{equation}

We now determine \(M(G)\) from \eqref{eq:PG-path}.

\textbf{Case 1.} If \(e\ge 1\), then \((1+x)^e\) vanishes at \(x=-1\), so
\[
P_G(-1)=P_{P_n}(-1)=a_n.
\]
If \(n\equiv 0,2\pmod 3\), then \(a_n\neq 0\) and hence \(M(G)=0=M(P_n)\). Then
   \[ 
    \mathfrak a (R/I(G))=\mathfrak a (R/I(P_n))=0.
   \]
If \(n\equiv 1\pmod 3\), then \(a_n=0\) and \(M(G)\ge 1\). Recall that $M(P_n)=1$ in this case. Our goal is to show $M(G)=1$. For this purpose, we compute \(P_G'(-1)\) by differentiating \eqref{eq:PG-path}.
\begin{itemize}
    \item If \(e\ge 2\) then  \(P_G'(-1)=P_{P_n}'(-1)=b_n\neq 0\).  Thus, \(M(G)=1=M(P_n)\) and
    \[ 
    \mathfrak a (R/I(G))= \mathfrak a (R/I(P_n))= -1.
    \]
    \item If \(e=1\), then 
\[
P_G'(-1)=b_n+\frac{d}{dx}\Bigl[x(1+x)(1+2x)^\ell\Bigr]\Big|_{x=-1}=b_n+(-1)^{\ell+1}.
\]
 Here \(n=3k+1\) and \(e=1\) forces \(n\ge 4\). Hence \(k\ge 1\) and \(|b_n|=k+1\ge 2\).
Therefore \(b_n+(-1)^{\ell+1}\neq 0\). Thus,  \(M(G)=1=M(P_n)\) and
   \[ 
    \mathfrak a (R/I(G))=\mathfrak a (R/I(P_n))=-1.
   \]
\end{itemize}
\textbf{Case 2.}  If \(e=0\), then \(n\equiv 1\pmod 3\) and \(\mathcal C=\{x_1,x_4,\ldots,x_{3k+1}\}\). This follows from the following discussion:
Recall from the proof of \Cref{prop:ell-plus-C}  that $|\mathcal C|=p+q+1$ and $ \ell=q$ where $p$ is the number of gaps
of type $01$ and  $q$ is the number of gaps of type $001$ between the first and last $1$. Thus
\[
e=|\mathcal C|-\ell-1+\delta=p+\delta\ge 0.
\]
In particular, \(e=0\) holds if and only if \(p=0\) and \(\delta=0\), i.e. all gaps have size \(3\) and
\(x_1,x_n\in\mathcal C\). When \(n\equiv 1\pmod 3\),  this forces
\(\mathcal C=\{x_1,x_4,\ldots,x_{3k+1}\}\).

 Evaluating \eqref{eq:PG-path} at \(-1\) gives
\[
P_G(-1)=P_{P_n}(-1)+(-1)\cdot(1+2(-1))^\ell
      =0+(-1)^{\ell+1}\neq 0.
\]
So \(M(G)=0\), while \(M(P_n)=1\). Hence, we have the following in the exceptional case
   \[ 
    \mathfrak a (R/I(G))=\mathfrak a (R/I(P_n))+1=0.
   \]
     This completes the proof. 
\end{proof}

\section{Closing remarks and further directions}\label{sec:final}

This paper studies how three homological invariants of edge ideals (projective dimension,
regularity, and the $\mathfrak a$-invariant) behave under \emph{suspension}, i.e. adjoining a new
vertex $z$ to a graph $H$ and joining it to a prescribed subset $\mathcal C\subseteq V(H)$.  Let $G=H(\mathcal{C})$ be the $\mathcal{C}$-suspension of $H$ where $R=\Bbbk [V(H)]$ and $S=R[z]$.

In the
families treated here, suspension over minimal vertex covers and over the maximal independent sets of
cycles (and, with one explicit exception, of paths) produces a remarkably rigid outcome:
regularity and the $\mathfrak a$-invariant are preserved, while projective dimension typically increases by
exactly one.  By contrast, suspension over all vertices preserves regularity but forces
$\pdim (S/I(G))=|V(H)|$.

A convenient starting point for broader generalizations is the short exact sequence obtained from
multiplication by $z$,
\[
0\longrightarrow S/(I(H\setminus \mathcal{C})+(x_i: x_i\in\mathcal C))(-1)\xrightarrow{\ \cdot z\ } S/I(G)
\longrightarrow S/(I(G),z)\cong R/I(H)\longrightarrow 0,
\]
together with the parallel decompositions of the independence complex (or equivalently, of
independence polynomials under deletion--contraction).  The results of this paper suggest that
under suitable hypotheses on $\mathcal C$ these constructions behave minimally, leading to controlled changes in $\reg, \pdim$ and $\mathfrak a$-invariant.

\begin{question}
For a graph $H$ and a subset $\mathcal C\subseteq V(H)$, give combinatorial conditions on $\mathcal C$
guaranteeing
\[
\pdim (S/I(G))=\pdim (R/I(H))+1.
\]
More generally, describe the set of possible values of
$\pdim (S/I(G))-\pdim (R/I(H))$ in terms of the interaction between $\mathcal C$ and the
 structure of $H$.
\end{question}

In the main families studied here, regularity is preserved under suspensions, except for a unique
extremal configuration for paths. 

\begin{question}
 Identify families of graphs  (e.g.\ trees, unicyclic graphs,
chordal graphs) for which suspension over every maximal independent set preserves regularity, and
characterize when exceptional suspension sets force $\reg$ to increase.  Is there a uniform upper bound
on $\reg S/I(G)-\reg R/I(H)$ in terms of classical invariants of $H$ and the structure of $\mathcal C$?
\end{question}

It was discussed in \cite{biermann2026realizableregdeghpairs} that $\mathfrak{a}(R/I(H)=-M(H)$ where $M(H)$ stands for the multiplicity of $x=-1$ as a root of the independence polynomial $P_H(x)$. Thus, our $\mathfrak a$-invariant computations ultimately reduce to finding $M(H)$. 

\begin{question}
 Develop a general criterion, stated purely in graph-theoretic terms, for when $\mathfrak a (S/I(G))=0$ (in other words,  $M(G)=0$)? 
 Similarly, find a criterion for when $\mathfrak a (S/I(G))=-1$.  In particular, when $P_G(-1)=0$,
under what hypotheses on $\mathcal C$ does one always have $P_G'(-1)\neq 0$, forcing
$\deg h_{S/I(G)}(t)=\alpha(G)-1$ and hence determining $\mathfrak a (S/I(G))$?
\end{question}

These questions point toward a broader goal suggested by our constructions:

\emph{Describe how a selection of $\mathcal C$ inside $H$ controls not only the  invariants $\pdim$ and $\reg$,
but also finer data such as the shape of the Betti table.}

\textbf{Acknowledgements.} The authors thank Des Martin, Hugh Geller, and Henry Potts--Rubin for inspiring this work. They are also grateful to Takayuki Hibi for helpful conversations that shaped the approach of the paper. The first author was supported by NSF grant DMS--2418805.

\bibliographystyle{abbrv}
\bibliography{ref}

\end{document}